\documentclass[3p,times]{elsarticle}

\usepackage{amsmath,amsthm}
\usepackage{amssymb,latexsym}
\usepackage{enumerate}

\numberwithin{equation}{section}

\newtheorem{theorem}{Theorem}[section]

\newtheorem{lemma}[theorem]{Lemma}

\newtheorem{The main theorem}[theorem]{The main theorem}

\theoremstyle{definition}

\allowdisplaybreaks

\begin{document}

\begin{frontmatter}

\title {Local H\"older continuity of  solutions of the  complex Monge-Amp\`ere equation}

\author[label1]{Nguyen Xuan Hong}
\address[label1]{Department of Mathematics, Hanoi National University of Education, 136 Xuan Thuy Street, Cau Giay District, Hanoi, Vietnam} \ead{xuanhongdhsp@yahoo.com}

\author[label3]{Pham Thi Lieu} \address[label3]{Department of Basis Sciences and Foreign Languages,
People's Police University of Technology and Logistics, Bacninh, Vietnam} 


\begin{abstract} 
In this paper, we are interested in studying the Dirichlet problem for the complex Monge-Amp\`ere operator. We provide necessary and sufficient conditions for the problem to have H\"older continuous solutions. 
\end{abstract}


\end{frontmatter}

\section{Introduction}
The Dirichlet problem for the complex Monge-Amp\`ere operator is one of the important and central problems of pluripotential theory. It has a long history, starting in the 70s of the 20th century.  
In this paper we are interested in the H\"older continuity of the solution of the problem. More precisely, we seek a plurisubharmonic function $u$ defined on a  bounded domain $\Omega$ of $\mathbb C^n$ satisfying: 
\begin{align*}
\mathcal M(\Omega, \mu ): \ 
\begin{cases}  
	u  \text{ is H\"older continuous on } \overline{\Omega} ; \\
	(dd^c u)^n = \mu \ \text{in } \Omega; \\
  u =0 \  \text{on }  \partial \Omega.
\end{cases}
\end{align*} 

There are some known results for the existence for this problem due to   \cite{BT76},  \cite{GKZ08},  \cite{Li04b}, etc.   In those papers, the authors only consider a simple case, that is, the measure $\mu$ only in 
$$
\mu_f:=fdV_{2n},\  f\geq 0 .
$$ 
The first result in this direction tells us that the problem $\mathcal M(\Omega, \mu_f )$ is solvable if $\Omega$ is strictly pseudoconvex and   $f^{1/n}$ is H\"older continuous (see \cite{BT76}). Therefore,  strictly pseudoconvex domains have been widely accepted as the standard domain in which we can analyze the Dirichlet problem for the complex Monge-Amp\`ere operator. This problem is further studied in \cite{BT82}, where the authors showed the existence of a continuous solution if $f$ is continuous.
Later, some other authors also generalized the result above. 
There is a very important result due to Guedj, Ko\l odziej and Zeriahi     that the problem is solvable for class of measure $\mu_f$ with $f\in L^p$, $p>1$ (see  \cite{GKZ08}). Recently, several other authors have used the technique of \cite{GKZ08} to study   the problem $\mathcal M(\Omega, \mu_f )$ for some subclasses of the class of pseudoconvex domains (see \cite{BKPZ},  \cite{Ch15},  \cite{HT}, etc).
In the case of manifolds, the problem is studied by \cite{DSHKZ}, \cite{Hie10} and some other authors. 

When $\mu$ is arbitrary, 
the problem becomes much more complicated. This problem remained open up until recently.
In the recent paper \cite{DGZ16}, Dinew, Guedj and   Zeriahi posed the question of the existence of H\"older continuous solutions of  Dirichlet problem (see Question 17 in \cite{DGZ16}). 
This question  was partially solved by Cuong \cite{Cuong18} in 2018 and fully completed in 2020 (see \cite{Cuong20}). 

The main purpose of this paper is to study the conditions which the problem can be solved. 
To study this, we will first find the  properties of the measure that have local sub-solutions. Here,  a function is called sub-solution to $\mathcal M (\Omega,\mu)$ if it is a solution to $\mathcal M (\Omega,\nu)$  with some measure $\nu\geq \mu$. Our first main result   is the following theorem.
\begin{theorem}
\label{thm1}
Let $\Omega$ be a bounded domain in $\mathbb C^n$ and let   $\mu$ be a non-negative Borel measure in $\Omega$. Asume that  for every $z\in \overline{\Omega}$, there exists $r_z>0$ such that   $\mathcal M(\Omega \cap \mathbb B(z,r_z), \mu )$  has a negative sub-solution.  Then, there exists a H\"older continuous function $u$ on $\overline{\Omega}$ such that it is plurisubharmonic in $\Omega$ and satisfies 
$$(dd^c u)^n \geq \mu \text{ on } \Omega.$$   
\end{theorem}

The above theorem provides us with a condition to be able to solve the problem. For the problem to be solved on the domain $\Omega$, we need to add the geometry of  $\Omega$, specifically, we require it to be strictly pseudoconvex. 
Our second main result   is the following theorem  on the locality of the problem.

\begin{theorem}
\label{thm2}
Let $\Omega$ be a  strictly pseudoconvex domain in $\mathbb C^n$ and let   $\mu$ be a non-negative Borel measure in $\Omega$. Then, the problem  $\mathcal M(\Omega, \mu )$  is local,
i.e., it is solvable on $\Omega$ if only if for every $z\in \overline{\Omega}$, there exists $r_z>0$ such that it is solvable on $\Omega\cap \mathbb B(z,r_z)$.
\end{theorem}

The remainder of this paper is organized as follows. In Section 2, we  prove an  auxiliary lemma and use it to prove Theorem \ref{thm1}. Section 3 is devoted to the proof of Theorem \ref{thm2}.

\section{Measure with local sub-solutions}
Some elements of pluripotential  theory  that will be used  throughout  the paper can be  found  in \cite{ACCH}-\cite{Kol98}.  To prove Theorem \ref{thm1}, we need the following result on  a special class of smooth increasing convex functions.  

\begin{lemma}
\label{l2}
Let $f_j:( -\infty, 0) \to ( - \infty, 0)$, $1\leq j \leq m$, be increasing  functions such that  
$$
\lim_{x\to 0} f_j(x) = 0.
$$
Then,   for every $\varepsilon\in (0, 1)$ there exists a smooth increasing convex function $\tau_\varepsilon : ( -\infty, -\varepsilon)\to \mathbb R$ such that 

(a)  $|\tau_\varepsilon \circ (f_j +x) -\tau _\varepsilon\circ (f_k +x)| <1 $  on $(-\infty,0)$, $\forall x\leq - \varepsilon$;

(b) $\tau_\varepsilon '(x)\geq \delta_\varepsilon>0$, $\forall x\in (-\frac{1}{\varepsilon}, -\varepsilon)$. 
\end{lemma}

\begin{proof} 
Let $\{a_j\}_{j=0}^\infty$ be  an increasing sequence of negative real numbers such that $a_j \to 0$ and 
\begin{align*}
b_{j}:& =\max_{1\leq k\leq m}  f_k(a_{j})
\\&   <\min _{1\leq k\leq m  } f_k(a_{j+1})=: c_{j+1}.
\end{align*}
Set  
$$
d_j:=
\begin{cases}
c_1-\frac{2}{\varepsilon} & \text{ if } j=1 \\
b_j& \text{ if } j \in \{2,4,6, \ldots\} \\
c_{j}& \text{ if } j \in \{3,5,7, \ldots\}.
\end{cases}
$$
We infer by the hypotheses that $d_j \nearrow 0$ as  $j \nearrow+\infty$. 
Let  $\tau: ( -\infty, 0)\to \mathbb R$ be  defined by  
$$
\tau  (x):=
\begin{cases}
n- \left( \frac{d_2}{d_1} +\cdots+\frac{d_n}{d_{n-1}}   \right)- \frac{x}{d_{n}}  & \text{ if } d_{n}<x\leq d_{n+1} ,\\
0 &\text{ if } x\leq d_1.
\end{cases}
$$ 
Proof of Lemma 2 in \cite{CM89} tells us that $\tau$ is continuous, increasing,  convex and satisfies 
\begin{equation}
\label{e4''}
\left| \tau  \circ f_j -\tau  \circ f_k \right|  \leq  A \text{ on } (-\infty,0), \ \forall 1\leq j,k \leq m.
\end{equation}
Here, $A>0$ is a constant.  
We now claim that 
\begin{equation}
\label{e7}
|\tau(x)-\tau(y)|\geq \left| \frac{ x-y }{d_1} \right| , \ \forall x,y\in (d_1,0).
\end{equation}
Indeed,  since $0<|d_n|\leq |d_1|$, we infer by the definition of $\tau$ that 
\begin{equation}
\label{e6}
\begin{split}
|\tau(x)-\tau(y)|  
& = \left| \frac{x-y}{d_n} \right| \\& \geq \left| \frac{x-y}{d_1} \right| , \ \forall x,y \in [d_n  , d_{n+1}].
\end{split}
\end{equation} 
Now, we assume that  
$$d_n<x\leq d_{n+1}$$
 and $$ d_{n+k} <y \leq d_{n+k+1}, \ k\geq 1.$$ 
Using the inequality  \eqref{e6} we have 
\begin{align*}
& \tau(y)-\tau(x)  
\\& = [\tau(y)-\tau(d_{n+k}) ] + \sum_{s=n+2}^{n+k}  [\tau(d_{s})-\tau(d_{s-1})]  + [\tau(d_{n+1})-\tau(x)] 
\\& \geq  \frac{ y- d_{n+k }}{-d_1}   +  \sum_{s=n+2}^{n+k}  \frac{d_{s} -d_{s-1}}{-d_1}    +  \frac{d_{n+1} - x  }{-d_1} 
\\& =\frac{y-x }{-d_1}.
\end{align*}
This proves the claim. 
Let $\rho$ be a smooth non-negative function in $\mathbb R$ such that $\rho(\lambda )=0$, $\forall |\lambda |\geq 1$ and 
$$
\int_{\{|\lambda|<1\} } \rho(\lambda ) d\lambda =1.	
$$
We define 
$$
\tau_\varepsilon(x):=  
\frac{1}{2A} \int_{\{|\lambda |<1\}} \tau(x-\varepsilon \lambda ) \rho(\lambda ) d\lambda , \   x \in (-\infty, -\varepsilon).
$$
Then, $\tau_\varepsilon$ is  smooth, increasing, convex on $(-\infty,-\varepsilon)$. 
Moreover, we deduce by \eqref{e7} that 
$$
|\tau_\varepsilon (x)-\tau_\varepsilon (y)|\geq \left| \frac{ x-y }{2A d_1} \right| , \ \forall x,y\in (d_1+\varepsilon,-\varepsilon).
$$
This implies that 
$$
\tau'_\varepsilon (x) \geq  \frac{ 1 }{2A d_1} =:\delta _\varepsilon , \ \forall x \in \left  (-\frac{1}{\varepsilon}, - \varepsilon \right )
$$
because 
$$d_1+\varepsilon<-\frac{1}{\varepsilon}.$$
Therefore, it remains to verify that  
\begin{equation}
\label{4'}
|\tau_\varepsilon \circ (f_j + x ) -\tau _\varepsilon\circ (f_k + x)| <1 , \ \forall x\leq -\varepsilon.
\end{equation}
Indeed, assume that  $|\lambda|<1$. It is clear that 
$$
x  -\varepsilon   \lambda \leq -\varepsilon( 1+\lambda) <0.
$$
From the convexity of $\tau$ we  infer by \eqref{e4''} that
\begin{align*}
& \left|\tau \circ (f_j +x  -\varepsilon   \lambda )   -\tau \circ (f_k +x  -\varepsilon   \lambda )\right| 
\\& \leq \left| \tau  \circ f_j -\tau  \circ f_k \right|  \leq A. 
\end{align*}
Hence,
\begin{align*}
& \left| \tau_\varepsilon \circ (f_j +x ) -\tau _\varepsilon\circ (f_k +x ) \right| 
\\& \leq  \frac{1}{2A} \int_{\{|\lambda|<1\}}    \left|  \tau(f_j +x  -\varepsilon \lambda ) -  \tau(f_k +x  -\varepsilon \lambda ) \right|   \rho(\lambda ) d\lambda
\\& \leq  \frac{1}{2} \int_{\{|\lambda|<1\}}     \rho(\lambda ) d\lambda < 1
\end{align*}
and \eqref{4'} is proved. 
This proves the lemma.
\end{proof}

We now able to give the proof of theorem \ref{thm1}. Techniques used  
come from  \cite{CM89}.

\begin{proof}[Proof of theorem \ref{thm1}] 
Since $\overline{\Omega}$ is a compact set,  we can find  $z_1,\ldots, z_m\in \overline \Omega$ and $r_1,\ldots, r_m>0$ such that $\overline{\Omega} \subset \bigcup_{j=1}^m \mathbb B(z_j, r_j)$ and   $\mathcal M (\Omega\cap \mathbb B(z_j, 3r_j) , \mu)$ has a  negative sub-solution  $u_j$. 
Define 
\begin{equation}
\label{e1''1}
u_j=0 \text{ on } \mathbb C^n \backslash (\Omega\cap \mathbb B(z_j,3 r_j)).
\end{equation} 
Then, $u_j$ are H\"older continuous on $\mathbb C^n$.  
For  $1\leq j,k \leq m$, we set $$G_{j,k}:=\Omega\cap \mathbb B(z_j, 3r_j) \cap \mathbb B(z_k,3r_k)
$$ 
and
$$f_{j,k}(x) =
\begin{cases} \inf\{u_j(z):z\in G_{j,k}\cap \{ u_k \geq x\} \} & \text{ if } G_{j,k} \neq \emptyset , 
\\
x & \text{ if } G_{j,k} = \emptyset.
\end{cases}
$$ 
It is easy to see that $f_{j,k}$ are increasing functions  and 
$$\lim_{x\to 0} f_{j,k} (x)= 0. $$   

\noindent 
Let $\varepsilon\in (0, 1)$ be such that 
\begin{equation}
\label{e1'1}
u_j >  - \frac{1}{\varepsilon}, \ \forall j=1,\ldots,m.
\end{equation}
Lemma \ref{l2} tells us that   there exists a smooth increasing convex function $\tau : ( -\infty, -\frac{\varepsilon}{2})\to \mathbb R$ such that:

(a)  $|\tau  (x-\varepsilon) -\tau (f_{j,k}(x) -\varepsilon) |< 1$, $\forall x<0$;  

(b) $ \tau'(x) \geq \delta>0$, $\forall x\in (-\frac{2}{\varepsilon } , -\frac{\varepsilon}{2})$. 

\noindent 
Assume that $z\in G_{j,k}$. Since $f_{j,k} (u_k(z)) \leq u_j(z)$, we deduce by (a) that 
\begin{equation}
\label{e4}
|\tau  \circ (u_j(z) -\varepsilon) -\tau  \circ (u_k (z)-\varepsilon) |< 1.
\end{equation}  

\noindent Let $\chi_1, \ldots, \chi_{2m}$ be smooth functions in $\mathbb C^n $  such that  
$$0\leq \chi_{2j-1} \leq \chi_{2j}  \leq 1$$  and 
\begin{align*}
\mathbb B  (z_j, r_j)
&  \Subset \{\chi_{2j-1} =1\}   \Subset \{\chi_{2j-1} \neq 0\} 
 \\&  \Subset\mathbb B  (z_j, 2r_j)  
\Subset \{\chi_{2j} =1\}
 \\&  \Subset \{\chi_{2j} \neq 0\} \Subset {\mathbb B} (z_j, 3r_j)  
\end{align*}
for all $j=1,\ldots, m$. We infer by   \eqref{e4}    that
\begin{equation*}
\label{e5}
\chi_{2j-1} + \chi_{2j} \tau(u_j-\varepsilon)<  \chi_{2k-1} + \chi_{2k} \tau(u_k-\varepsilon)
\end{equation*} 
on $\Omega  \cap \mathbb B(z_k, r_k) \cap (\mathbb B(z_j, 3r_j ) \backslash  \mathbb B(z_j, 2r_j))$.  Hence,  
\begin{equation}
\label{e6'}
\chi_{2j-1} + \chi_{2j} \tau(u_j-\varepsilon)< \chi_{2k-1} + \chi_{2k} \tau(u_k-\varepsilon) 
\end{equation}  

\noindent 
on   $\mathbb B(z_k, r_k)  \backslash  \mathbb B(z_j, 2r_j)$.
Let $\varphi$ be a smooth strictly plurisubharmonic function in $\mathbb C^n$ such that 
$\varphi+\chi_{2j-1}$   is  plurisubharmonic in $\mathbb C^n$ for all $j=1, \ldots, m$. We define 
$$v:=\max_{1\leq j\leq m} v_j \ \text{ in } \mathbb C^n,$$ 
where  
$$v_j:=   \varphi+ \chi_{2j-1} + \chi_{2j} \tau(u_j-\varepsilon).
$$
Now, for $z\in \Omega \cap \mathbb B(z_k,r_k)$, we set 
$$
J_k(z) = \{ j\in \{1,\ldots,m\}:   z\in \mathbb B(z_j, 2 r_j)\}.
$$ 
From \eqref{e6'} we have 
$$v_j(z)<v_k(z), \ \forall j\not \in J_k(z).$$
This implies that 
$$
U_k (z):= \Omega \cap \mathbb B(z_k,r_k) \cap \bigcap_{j\in J_k(z)}  \mathbb B(z_j,2r_j) \cap \bigcap_{1\leq j\leq m, \ j\not \in J_k(z)} \{v_j<v_k\}
$$
is a open neighborhood of $z$. 
Since $\varphi+ \chi_{2j-1}$ are plurisubharmonic  and   $$
v_j = \varphi+ \chi_{2j-1} +\tau(u_j-\varepsilon) \text{  on }  U_k(z),
$$ 
it follows that  $v_j \in PSH(  U_k(z))$ and 
\begin{align*}
(dd^c v_j)^n & \geq (dd^c \tau( u_j-\varepsilon))^n
\\&  \geq [\tau'(u_j-\varepsilon)]^n (dd^c u_j)^n  
\text{ on } U_k(z).
\end{align*}
Therefore, we deduce from  (b)  that 
$v \in PSH(U_k(z))$
and 
\begin{equation}
\label{e8}
(dd^c v)^n  \geq \delta^n  \mu  \text{ on } U_k(z)
\end{equation}
because 
$$u_j-\varepsilon> -\frac{2}{\varepsilon} \text{ on }U_k(z).$$  
Now, since  $\tau$ is smooth on $(-\frac{2}{\varepsilon}, -\frac{\varepsilon}{2})$, we obtain by \eqref{e1''1} and \eqref{e1'1} that 
$$\tau\circ (u_j-\varepsilon) \text{ are  H\"older continuous in } \mathbb C^n.$$   
Combining this with  \eqref{e8} we arrive  that   $v$ is H\"older continuous on $\overline{\Omega}$, plurisubharmonic on $\Omega$ and satisfies 
$$
(dd^c v)^n \geq  \delta^n  \mu \text{  on } \Omega.
$$
Then, function $u:=  \delta  v$ satisfies the requirements of the theorem, and the proof is complete. 
\end{proof}

\section{Local H\"older continuity}

\begin{proof}
[Proof of theorem \ref{thm2}]
The necessity is obvious. We prove the sufficiency. 
Since $\Omega$ is strictly pseudoconvex, by Theorem B in \cite{Cuong20}, 
we need to prove that $\mathcal M(\Omega,\mu)$ has a sub-solution. 
To achieve this, we will use Theorem \ref{thm1}, so we have to show that there exist local sub-solutions. 
Let $\rho$ be   a  H\"older continuous function  on $\overline{\Omega}$ such that it is negative, plurisubharmonic in $\Omega$ and satisfies 
$$
\rho=0 \text{ on } \partial \Omega.
$$
Let $z\in \overline{\Omega}$ and define
$$
\rho_z(t) := \max( \rho(t), \|t-z\|^2- r_z^2), \ t\in \overline \Omega.
$$
Then, $\rho_z$ is H\"older continuous on $\overline{\Omega}$ and  plurisubharmonic on $\Omega$. This implies that $\rho_z$ is a negative  sub-solution of $\mathcal M ( \Omega \cap \mathbb B(z, r_z), 0)$.  
Let $u_z$ is a solution of $\mathcal M ( \Omega \cap \mathbb B(z, r_z), \mu)$.
The maximum principle tells us that  
$$
u_z\leq 0 \text{ on }  \Omega \cap \mathbb B(z, r_z),
$$
and hence, 
$$\varphi_z:= \rho_z+u_z$$ 
is a negative sub-solution of $\mathcal M ( \Omega \cap \mathbb B(z, r_z), \mu)$.
Theorem \ref{thm1} states that there exists a H\"older continuous function $\varphi $ on $\overline{\Omega}$ such that it is plurisubharmonic in $\Omega$ and satisfies 
$$(dd^c \varphi)^n \geq \mu \text{ on } \Omega.$$   
On the other hand, since $\Omega$ is strictly pseudoconvex, Theorem 1.2 in \cite{Li04b} shows that there is a H\"older continuous function $\psi$ on $\overline{\Omega}$ such that it is plurisubharmonic in $\Omega$ and 
$$
\psi=-\varphi \text{ on } \partial \Omega.
$$
It is easy to see that $\varphi+\psi$ is a sub-solution of $\mathcal M(\Omega, \mu)$, and therefore, we conclude by  Theorem B in \cite{Cuong20} that  $\mathcal M (\Omega, \mu)$ is solvable.  
The proof is complete. 
\end{proof}


\end{document}